\documentclass[12pt,letterpaper]{amsart}
\usepackage{geometry}
\usepackage{mathpazo,setspace,xcolor} 
\usepackage{hyperref}
\usepackage{amsmath,amssymb,amsfonts,amsthm}
\newtheorem{theorem}{Theorem}

\newtheorem{lemma}{Lemma}

\theoremstyle{remark}

\newtheorem{remark}{Remark}
\newcommand{\C}{\mathbb{C}}
\newcommand{\D}{\Omega}
\newcommand{\ep}{\varepsilon}
\newcommand{\Dc}{\overline{\Omega}}
\newcommand{\dbar}{\overline{\partial}}
\newcommand{\dbars}{\overline{\partial}^{\ast}}
\newcommand{\zb}{\overline{z}}
\newcommand{\wb}{\overline{w}}
\onehalfspace
\title[Compactness of Hankel and $\dbar$-Neumann operators]{On 
	compactness of Hankel  and the $\dbar$-Neumann operators on 
	Hartogs domains in $\C^2$}
\author{S\"{o}nmez \c{S}ahuto\u{g}lu}
\address[S\"{o}nmez \c{S}ahuto\u{g}lu]{University of Toledo, Department of 
Mathematics \& Statistics, Toledo, OH 43606, USA}
\email{Sonmez.Sahutoglu@utoledo.edu}

\author{Yunus E. Zeytuncu}
\address[Yunus E. Zeytuncu]{University of Michigan--Dearborn, Department of 
Mathematics \& Statistics, Dearborn, MI 48128, USA}
\email{zeytuncu@umich.edu}
\thanks{The work of the second author was partially supported by a grant from 
the Simons Foundation (\#353525), and also by a University of 
Michigan-Dearborn CASL Faculty Summer Research Grant. 
Both authors would like to thank the American Institute of Mathematics 
for hospitality for hosting a workshop during which this project was started.}
\begin{document}
\begin{abstract}
	We prove that on smooth bounded pseudoconvex Hartogs domains in $\C^2$ 
	compactness of the $\dbar$-Neumann operator is equivalent to compactness 
	of all Hankel operators with symbols smooth on the closure of the domain.  
\end{abstract}
\maketitle

Let $\Omega$ be a bounded pseudoconvex domain in $\mathbb{C}^n$ and 
 $L^2_{(0,q)}(\Omega)$ denote the space of square integrable $(0,q)$ forms 
for $0\leq q\leq n$. The complex Laplacian $\Box=\dbar\dbars+\dbars\dbar$ is 
a densely defined, closed,  self-adjoint linear  operator on $L^2_{(0,q)}(\Omega)$.  
H\"{o}rmander  in \cite{Hormander65} showed that  when $\D$ is bounded and 
pseudoconvex, $\Box$ has a bounded solution operator $N_q$, called the 
$\dbar$-Neumann operator for all $q$. Kohn in \cite{Kohn63} showed that the 
Bergman projection, denoted by $\mathbf{B}$ below,   is connected to the 
$\dbar$-Neumann operator via the following formula 
\[\mathbf{B}=\mathbf{I}-\dbar^*N_1\dbar\]
where $\mathbf{I}$ denotes the identity operator. For more information 
about the $\dbar$-Neumann problem we refer the reader to two books  
\cite{ChenShawBook,StraubeBook}.

Let $A^2(\Omega)$ denote the space of square integrable holomorphic functions 
on $\D$ and $\phi\in L^{\infty}(\Omega)$. The Hankel operator with symbol 
 $\phi, H_{\phi}: A^2(\Omega)\to L^2({\Omega})$  is defined by 
\[H_{\phi}g=[\phi,\mathbf{B}]g=\left(\mathbf{I}-\mathbf{B}\right)\left(\phi g\right).\]
Using Kohn's formula one can immediately see that 
 \[H_{\phi}g=\dbar^*N_1( g\dbar \phi)\]
 for $\phi \in C^1(\Dc)$.
It is clear that $H_{\phi}$ is a bounded operator; however, its compactness depends 
on both the function theoretic properties of the symbol $\phi$ as well as the geometry 
of the boundary of the domain $\Omega$ (see \cite{CuckovicSahutoglu09}). 

The following observation is relevant to our work here. Let $\D$ be a bounded 
pseudoconvex domain in $\C^n$ and $\phi\in C(\overline{\Omega})$. If 
$\dbar$-Neumann operator $N_1$ is compact on $L^2_{(0,1)}(\Omega)$ then the 
Hankel operator $H_{\phi}$ is compact (see \cite[Proposition 4.1]{StraubeBook}). 

We are interested in the converse of this observation.  Namely, 
\begin{quote}
\textit{Assume that $\D$ is a bounded pseudoconvex domain in $\C^n$ and 
	$H_{\phi}$ is compact on $A^2(\Omega)$ for all symbols $\phi\in C(\overline{\Omega})$. 
	Then is the $\dbar$-Neumann operator $N_1$ compact on $L^2_{(0,1)}(\D)$?} 
\end{quote}
This is known as D'Angelo's question and has first appeared in 
\cite[Remark 2]{FuStraube01}.

The answer to D'Angelo's question is still open in general but there are some partial 
results. Fu and Straube in \cite{FuStraube98} showed that the answer is yes if $
\Omega$ is convex. \c{C}elik and the first author \cite[Corollary 1]{CelikSahutoglu2012} 
observed that if $\D$ is not pseudoconvex then the answer to D'Angelo's question may 
be no. Indeed, they  constructed an annulus type domain $\Omega$ where $H_{\phi}$ 
is compact on $A^2(\Omega)$ for all symbols $\phi\in C(\overline{\Omega})$; yet, the 
$\dbar$-Neumann operator $N_1$ is not compact on $L^2_{(0,1)}(\Omega)$. 

\begin{remark}
One can extend the definition of Hankel operators from holomorphic functions to the 
$\dbar$-closed $(0,q)$ forms (denoted by $K^2_{(0,q)}(\Omega)$) and ask the analogous 
problem at the forms level. In this case, an affirmative answer was obtained in 
\cite{CelikSahutoglu14}. Namely, for $1\leq q\leq n-1$ if  $H^q_{\phi}=[\phi, \mathbf{B}_q]$ 
is compact on $K^2_{(0,q)}(\Omega)$ for all symbols $\phi\in C^{\infty}(\overline{\Omega})$ 
then the $\dbar$-Neumann operator $N_{q+1}$ is compact on $L^2_{(0,q)}(\Omega)$.
\end{remark}

In this paper, we provide an affirmative answer to D'Angelo's question on smooth 
bounded pseudoconvex Hartogs domains in $\mathbb{C}^2$. 

\begin{theorem}\label{Thm1}
 Let $\D$ be a smooth bounded pseudoconvex Hartogs domain in $\C^2$. The 
 $\dbar$-Neumann operator $N_1$ is compact  on $L^2_{(0,1)}(\D)$ if and only if 
 $H_{\psi}$ is compact on $A^2(\D)$ for all  $\psi\in C^{\infty}(\Dc)$.
\end{theorem}

As mentioned above, compactness of $N_1$ implies that $H_{\psi}$ is compact on any 
bounded pseudoconvex domain  (see \cite[Proposition 4]{FuStraube01} or  
\cite[Proposition 4.1]{StraubeBook}). The key ingredient of our proof of the converse 
is the characterization of the compactness of $N_1$ in terms of ground state energies 
of certain Schr\"odinger operators as previously explored in \cite{FuStraube02,FuChrist05}. 

We will need few lemmas before we prove Theorem \ref{Thm1}. 

\begin{lemma}\label{LemSobolev1}
Let $A(a,b)=\{z\in \C:a<|z|<b\}$ for $0<a<b<\infty$ and $d_{ab}(w)$ be the distance from 
$w$ to the boundary of $A(a,b)$. Then  there exists $C>0$ such that 
\[\int_{A(a,b)}(d_{ab}(w))^2|w|^{2n}dV(w)\leq \frac{C}{n^2} \int_{A(a,b)}|w|^{2n}dV(w)\]
for nonzero integer $n$.
\end{lemma}
\begin{proof}
We will use the fact that $d_{ab}(w)=\min\{b-|w|,|w|-a\}$ with polar coordinates to 
compute the first integral. One can compute that 
\[\int_{A(a,b)}|w|^{2n}dV(w)= \frac{\pi}{n+1}(b^{2n+2}-a^{2n+2})\]
for $n\neq -1$.  Let $c=\frac{a+b}{2}$. Then 
\begin{align*}
\int_{A(a,b)}(d_{ab}(w))^2|w|^{2n}dV(w)
=& \int_{A(a,c)}(|w|-a)^2|w|^{2n}dV(w)\\
&+\int_{A(c,b)}(b-|w|)^2|w|^{2n}dV(w)\\
 =&2\pi\int_a^c(a^2\rho^{2n+1}-2a\rho^{2n+2}+\rho^{2n+3})d\rho\\
 &+2\pi\int_c^b(b^2\rho^{2n+1}-2b\rho^{2n+2}+\rho^{2n+3})d\rho\\
 =&2\pi(b^{2n+4}-a^{2n+4})\left(\frac{1}{2n+2}-\frac{2}{2n+3}+\frac{1}{2n+4}\right)\\
 &+2\pi(a^2-b^2)\frac{c^{2n+2}}{2n+2}+4\pi(b-a)\frac{c^{2n+3}}{2n+3}\\
 =&\frac{\pi(b^{2n+4}-a^{2n+4})}{(n+1)(n+2)(2n+3)}-\frac{\pi c^{2n+2}(b^2-a^2)}{(n+1)(2n+3)}.
\end{align*}
In the last equality we used  the fact that $c=\frac{a+b}{2}$. Then one can show that 
\[\lim_{n\to\pm \infty}\frac{n^2\int_{A(a,b)}(d_{ab}(w))^2|w|^{2n}dV(w)}{\int_{A(a,b)}|w|^{2n}dV(w)}
=\frac{b^2}{2}.\]
Therefore, there exists $C>0$ such that 
\[\int_{A(a,b)}(d_{ab}(w))^2|w|^{2n}dV(w)\leq \frac{C}{n^2} \int_{A(a,b)}|w|^{2n}dV(w)\]
 for nonzero integer $n$.
\end{proof}

We note that throughout the paper $\|.\|_{-1}$ denotes the Sobolev $-1$ norm. 

\begin{lemma}\label{LemSobolev2}
Let  $\D=\{(z,w)\in \C^2: z\in D \text{ and } \phi_1(z)<|w|<\phi_2(z)\}$ be a 
bounded Hartogs domain.  Then there exists $C>0$ such that 
 \[\|g(z)w^n\|_{-1}\leq \frac{C}{n}\|g(z)w^n\|\]
for any  $g\in L^2(D)$ and nonzero integer $n$, as long as the right hand side is finite.
\end{lemma}
 \begin{proof}
We will denote the distance from $(z,w)$ to the boundary of $\D$ by $d_{\D}(z,w)$. 
 	 We note that $W^{-1}(\D)$ is the dual of $W^1_0(\D)$, the closure of $C^{\infty}_0(\D)$ 
 	 in $W^1(\D)$. Furthermore, 
 	 \[\|f\|_{-1}=\sup\{|\langle f,\phi\rangle|:\phi\in C^{\infty}_0(\D), \|\phi\|_1\leq 1\}\]
 	 for $f\in W^{-1}(\D)$. Then there exists $C_1>0$ such that 
 	 \begin{align*}
 	 \|f\|_{-1}\leq \|d_{\D}f\|\sup\{\|\phi/d_{\D}\|:\phi\in C^{\infty}_0(\D), \|\phi\|_1\leq 1\}
 	 \leq C_1\|d_{\D}f\|. 
 	 \end{align*}
In the second inequality above we used the fact that 
(see  \cite[Proof of Theorem C.3]{ChenShawBook})
there exists $C_1>0$ such that 	 
$\|\phi/d_{\D}\|\leq C_1\|\phi\|_1$  for all $\phi \in W^1_0(\D)$.
 	 
Let $d_z(w)$ denote the distance from $w$  to the 
boundary of $A(\phi_1(z),\phi_2(z))$. Then 
there exists $C_1>0$ such that  
\begin{align*}
 \|g(z)w^n\|_{-1}^2\leq&C_1\int_{\D}(d_{\D}(z,w))^2|g(z)|^2|w|^{2n}dV(z,w)\\
 \leq &C_1\int_D|g(z)|^2\int_{\phi_1(z)<|w|<\phi_2(z)}(d_z(w))^2|w|^{2n}dV(w).
 \end{align*}
Lemma \ref{LemSobolev1} and the assumption that $\D$ is bounded    
imply that there exists $C_2>0$ such that 
 \[\int_{\phi_1(z)<|w|<\phi_2(z)}(d_z(w))^2|w|^{2n}dV(w)\leq 
\frac{C_2}{n^2}\int_{\phi_1(z)<|w|<\phi_2(z)}|w|^{2n}dV(w).\]
Then 
 \begin{align*}
 \int_D|g(z)|^2\int_{\phi_1(z)<|w|<\phi_2(z)}(d_z(w))^2|w|^{2n}dV(w)
 \leq &\frac{C_2}{n^2}\int_D|g(z)|^2\int_{\phi_1(z)<|w|<\phi_2(z)}|w|^{2n}dV(w)\\
 =&\frac{C_2}{n^2} \|g(z)w^n\|^2.
\end{align*}
Therefore, for $C=\sqrt{C_1C_2}$ we have  
$\|g(z)w^n\|_{-1}\leq \frac{C}{n}\|g(z)w^n\|$ for nonzero integer $n$.
\end{proof}

\begin{lemma}\label{LemCompEstimate}
	Let $\D$ be a bounded pseudoconvex domain in $\C^n$ and $\psi\in C^1(\Dc)$. 
Then $H_{\psi}$ is compact if and only if  for any $\ep>0$ there exists 
	$C_{\ep}>0$ such that 
	\begin{align}\label{EqnCompEst}
		\|H_{\psi}h\|^2\leq \ep\|h\dbar 
		\psi\|\|h\|+C_{\ep}\|h\dbar\psi\|_{-1}\|h\| 
	\end{align}
	for  $h\in A^2(\D)$.
\end{lemma}
\begin{proof}
	First assume that $H_{\psi}$ is compact. Then  
	\[\| H_{\psi}h\|^2= \langle H^*_{\psi}H_{\psi} h,h\rangle \leq \|H^*_{\psi}H_{\psi} h\| \|h\|\]
	for  $h\in A^2(\D)$. Compactness of $H_{\psi}$ implies that $H^*_{\psi}$ is compact. 
		Now we apply the compactness estimate in \cite[Proposition V.2.3]{D'AngeloBook}
		to  $H^*_{\psi}$. For $\ep>0$ there exists a compact operator $K_{\ep}$ 
		such that 
		\begin{align*}
			\|H^*_{\psi}H_{\psi} h\|  \leq & \frac{\ep}{2\|\dbar^*N\|} 
			\|H_{\psi} h\| + \|K_{\ep}H_{\psi} h\| \\
			\leq & \frac{\ep}{2} \| h\dbar \psi \|+ \|K_{\ep}H_{\psi} h\|.
		\end{align*}
		In the second inequality we used the fact that 
		$H_{\psi}h=\dbar^*N(h\dbar \psi)$. Since $\D$ is bounded pseudoconvex 
		$\dbar^*N$ is bounded and hence $K_{\ep}\dbar^*N$ is compact. 
		Now we use the fact that $H_{\psi}h=\dbar^*N(h\dbar \psi)$ and  
		 \cite[Lemma 4.3]{StraubeBook} for the compact operator 
		$K_{\ep}\dbar^*N$ to conclude that there exists 
		$C_{\ep}>0$ such that   
	\[\|K_{\ep}H_{\psi} h\|\leq \frac{\ep}{2}\|h\dbar \psi\| 
		+C_{\ep}\|h\dbar\psi\|_{-1}.\]
		Therefore, for $\ep>0$ there exists $C_{\ep}>0$ such that 
		\[\| H_{\psi}h\|^2\leq \ep\|h\dbar \psi\|\|h\| +C_{\ep}\|h\dbar\psi\|_{-1}\|h\|\]
			for  $h\in A^2(\D)$.
	
	To prove the converse assume \eqref{EqnCompEst} and choose $\{h_j\}$ a sequence 
	in $A^2(\D)$ such that  $\{h_j\}$ converges to zero weakly. Then the sequence 
	$\{h_j\}$ is bounded and $\|h_j\dbar\psi\|_{-1}$ converges to 0 (as the imbbeding 
	from $L^2$ into Sobolev $-1$ is compact). The inequality  	\eqref{EqnCompEst} 
	implies that there exists $C>0$ such that for every $\ep>0$ there exists 
	$J$ such that $\|H_{\psi}h_j\|^2 \leq C\ep$ for $ j\geq J$. 
	That is, $\{H_{\psi}h_j\}$ converges to 0. That is, $H_{\psi}$ is compact. 	
	\end{proof}

The following lemma is contained in \cite[Remark 1]{Sahutoglu12}. The 
superscripts on the Hankel operators are used to emphasize the domains. 
\begin{lemma}[\cite{Sahutoglu12}]\label{LemLocalization}
	Let $\D_1$ be a bounded pseudoconvex domain in $\C^n$ and  $\D_2$ be a  
	bounded strongly pseudoconvex domain in $\C^n$ with $C^2$-smooth boundary. 
	Assume that $U=\D_1\cap \D_2$ is connected, $\phi \in C^1(\Dc_1)$, and 
	$H^{\D_1}_{\phi}$  is compact  on $A^{2}(\D_1).$ Then $H^{U}_{\phi}$ is 
	compact on $A^{2}(U).$
\end{lemma}
 Now we are ready to prove Theorem \ref{Thm1}. 
 
\begin{proof}[Proof of Theorem \ref{Thm1}]
We present proof of the nontrivial direction. That is, we assume that $H_{\psi}$ 
 is compact on $A^2(\D)$ for all  $\psi\in C^{\infty}(\Dc)$ and prove that $N_1$ 
 is compact. Our proof is along the lines of the proof of \cite[Theorem 1.1]{FuChrist05}. 
 
Let $\rho(z,w)$ be a smooth defining function for $\D$ that is invariant under 
rotations in $w$. That is, $\rho(z,w)=\rho(z,|w|)$,  
\[\D= \{(z, w) \in \C^2: \rho(z, w) < 0\},\] 
and $\nabla \rho$ is nonvanishing on $b\D$.  Let 
$\Gamma_0 = \{(z, w) \in b\D : \rho_{|w|}(z, |w|)=0\}$ 
and \[\Gamma_k = \{(z, w) \in b\D: |\rho_{|w|}(z, |w|)|\geq 1/k\}\] 
for $k=1,2,\ldots$. We will show that 
$\Gamma_k$ is $B$-regular for $k=0,1,2,\ldots$ by establishing the estimates 
\eqref{estimate1} and \eqref{estimate2} below and invoking 
\cite[Lemma 10.2]{FuChrist05}. Then 
\[b\D=\bigcup_{k=0}^{\infty}\Gamma_k\] 
and  \cite[Proposition 1.9]{Sibony87} implies that $b\D$ is $B$-regular (satisfies 
Property $(P)$ in Catlin's terminology). This will be enough to conclude that 
$N_1$ is compact  on $L^2_{(0,1)}(\D)$

The proof of the fact that $\Gamma_0$ is $B$-regular is essentially contained in 
\cite[Lemma 10.1]{FuChrist05} together with the following fact: Let $\D$ be a 
smooth bounded pseudoconvex domain in $\C^2$. If $H_{\zb}$ and $H_{\wb}$ 
are  compact on $A^2(\D)$ then there is no analytic disc 
in $b\D$ (see \cite[Corollary 1]{CuckovicSahutoglu09}). 

Now we will prove that $\Gamma_k$ is $B$-regular for any fixed $k\geq 1$. Let 
$(z_0,w_0)\in \Gamma_k$, we argue in two cases. The first case is when   
$\rho_{|w|}(z_0,|w_0|)<0$ and the second case is $\rho_{|w|}(z_0,|w_0|)>0$.

We continue with the first case. Assume that $b\D$ near $(z_0,w_0)$ is given 
by $|w|=e^{-\varphi(z)}$. Let $D(z_0,r)$ denote the disc centered at $z_0$ with 
radius $r$ and 
\[U_{a,b}=D(z_0,a)\times \{w\in \C:|w_0|-b<|w|<|w_0|+b\}\] 
for $a,b>0$. Then let us choose $a,a_1,b,b_1>0$ such that $a_1>a,b_1>|w_0|+b$, the 
open sets 
\begin{align*}
U=\D\cap U_{a,b}= \left\{(z,w)\in \C^2:z\in D(z_0,a), ~e^{-\varphi(z)}<|w|<|w_0|+b\right\}
\end{align*}
and  $U_1=\D\cap U^{a_1,b_1}$ are connected where   
\[U^{a_1,b_1}=\left\{(z,w)\in \C^2:\frac{|z-z_0|^2}{a_1^2}+\frac{|w|^2}{b_1^2}<1\right\},\]
and finally $\overline{U}\subset U_1$. 
Then \[U_1=\left\{(z,w)\in \C^2:z\in V_1, e^{-\varphi(z)}<|w|<e^{-\alpha(z)}\right\}\]
where $V_1$ is a domain in $\C$ such that 
$\overline{D(z_0,a)}\subset V_1\subset D(z_0,a_1)$ and 
\[\alpha(z)=\log a_1-\log b_1-\frac{1}{2}\log(a_1^2-|z-z_0|^2).\]
One can check that $\alpha$ is subharmonic on $D(z_0,a_1)$, while pseudoconvexity 
of $\D$ implies that the function $\varphi$ is superharmonic on $D(z_0,a_1)$. 
Furthermore, since $B$-regularity is invariant under holomorphic change of 	
coordinates, by mapping under $(z,w)\to (z,\lambda w)$ for some $\lambda>1$, 
we may assume that  
	\[U_1\subset D(z_0,a_1)\times\{w\in\C:|w|>1\}.\]
For any $\beta\in C^{\infty}_0(D(z_0,a))$ let us choose $\psi\in C^{\infty}(\overline{V_1})$ 
such that $\psi_{\zb}=\beta$. Lemma \ref{LemLocalization} implies that  the 
Hankel operator $H^{U_1}_{\psi}$ (we use the superscript $U_1$ to emphasize the 
domain) is compact on the Bergman space $A^2(U_1)$. 

Let 
\[\lambda_n(z)=-\log\left(\frac{\pi}{n-1}\left(e^{(2n-2)\varphi(z)}
-e^{(2n-2)\alpha(z)}\right)\right)\] 
for $n=2,3,\ldots$.  One can check that since $\varphi$ is superharmonic and $\alpha$ is 
subharmonic, the function $ \lambda_n$ is subharmonic. Let $S^{V_1}_{\lambda_n}$ 
be the canonical solution operator for $\dbar$ on $L^2(V_1,\lambda_n)$. 
If $f_n=H^{U_1}_{\psi}w^{-n}$ then we claim that 
\[f_n(z,w)=g_n(z)w^{-n}\] 
where  $g_n=S^{V_1}_{\lambda_n}(\beta d\zb)$ and $n=2,3,\ldots$. Clearly 
$H^{U_1}_{\psi}w^{-n}=f_n\in L^2(U_1)$ and 
\[\dbar g_n(z)w^{-n}=\beta(z)w^{-n}d\zb.\] 
 
 To prove the claim we will just need to show that $g_n(z)w^{-n}$ is 
 orthogonal to $A^2(U_1)$. 
 That is, we need to show that $\langle g_n(z)w^{-n}, h(z)w^m\rangle_{U_1}=0$ 
 for any $h(z)\in A^2(V_1)$ and $m\in \mathbb{Z}$. Then 
\begin{align*}
\langle g_n(z)w^{-n}, h(z)w^m\rangle_{U_1}
&=\int_{U_1}g_n(z)w^{-n} \overline{h(z)w^m}dV(z)dV(w)\\
&=\int_{V_1}g_n(z)\overline{h(z)}dV(z) 
\int_{e^{-\varphi(z)}<|w|<e^{-\alpha(z)}}w^{-n}\overline{w^m}dV(w). 
\end{align*}
Unless $m=-n$ the integral 
$\int_{e^{-\varphi(z)}<|w|<e^{-\alpha(z)}}w^{-n}\overline{w^m}dV(w)=0$. 
So let us assume that $m=-n$. In that case we get
\[\int_{V_1}g_n(z)\overline{h(z)}dV(z) 
\int_{e^{-\varphi(z)}<|w|<e^{-\alpha(z)}}w^{-n}\overline{w^m}dV(w)
=\int_{V_1}g_n(z)\overline{h(z)}e^{-\lambda_n(z)}dV(z).\]
The integral on the right hand side above is zero because $g_n$ is orthogonal 
to $A^2(V_1,\lambda_n)$. Therefore, 
\[g_n(z)w^{-n}=H^{U_1}_{\psi}w^{-n}.\]
The equality above implies that 
$\frac{\partial g_n}{\partial \zb}=\frac{\partial \psi}{\partial \zb}=\beta$.
Then the compactness estimate \eqref{EqnCompEst} implies that  
\begin{align*}
\int_{D(z_0,a)}|g_n(z)|^2e^{-\lambda_n(z)}dV(z) 
\leq& \|g_n(z)w^{-n}\|_{U_1}^2\\
\leq& \ep \|\beta(z)w^{-n}\|_{U_1}\|w^{-n}\|_{U_1}+ 
C_{\ep}\|\beta(z)w^{-n}\|_{W^{-1}(U_1)}\|w^{-n}\|_{U_1} \\
=&\ep\left(\int_{D(z_0,a)}|\beta(z)|^2e^{-\lambda_n(z)}dV(z)\right)^{1/2}
\left(\int_{V_1} e^{-\lambda_n(z)}dV(z)\right)^{1/2} \\
&+ C_{\ep} \|\beta(z)w^{-n}\|_{W^{-1}(U_1)}
\left(\int_{V_1} e^{-\lambda_n(z)}\right)^{1/2}.
\end{align*}
Then by Lemma \ref{LemSobolev2} there exists $C>0$ such that 
\[\|\beta(z)w^{-n}\|_{W^{-1}({U_1})}\leq \frac{C}{n}\|\beta(z)w^{-n}\|_{U_1}
= \frac{C}{n}\|\beta\|_{L^2(D(z_0,a),\lambda_n)}.\] 
We note that to get the equality above we used the fact that $\beta$ is 
supported in $D(z_0,a)$. Hence we get 
\[\|g_n\|^2_{L^2(D(z_0,a),\lambda_n)}
\leq \left(\ep+\frac{CC_{\ep}}{n} \right) 
\|\beta\|_{L^2(D(z_0,a),\lambda_n)}\|1\|_{L^2(V_1,\lambda_n)}.\]
For any $\ep>0$ there exists an integer $n_{\ep}$ such that  
\[\frac{CC_{\ep}}{n}+ \frac{\pi a_1}{\sqrt{n-1}}\leq \ep\] 
for $n\geq n_{\ep}$. Then 
 \[\|g_n\|_{L^2(D(z_0,a),\lambda_n)}^2\leq 2\ep \|\beta\|_{L^2(D(z_0,a),\lambda_n)}  
\|1\|_{L^2(V_1,\lambda_n)}
\leq 2\ep^2 \|\beta\|_{L^2(D(z_0,a),\lambda_n)}\] 
for $n\geq n_{\ep}$ because $U\subset D(z_0,a)\times\{w\in\C:|w|>1\}$ and  
\begin{align*}
\|1\|_{L^2(V_1,\lambda_n)}&\leq \|1\|_{L^2(D(z_0,a_1),\lambda_n)}\\ &=\left(\int_{D(z_0,a_1)}\frac{\pi}{n-1}\left(e^{(2n-2)\varphi(z)}
-e^{(2n-2)\alpha(z)}\right)dV(z)\right)^{1/2}\\ 
&\leq \left(\int_{D(z_0,a_1)}\frac{\pi}{n-1}dV(z)\right)^{1/2}\\
&= \frac{\pi a_1}{\sqrt{n-1}}.
\end{align*}
Let $u\in C^{\infty}_0(D(z_0,a))$ and $n\geq n_{\ep}$. Then
\begin{align*}
 \int_{D(z_0,a)}|u(z)|^2e^{\lambda_n(z)}dV(z)=
 & \sup\left\{ |\langle u,\beta\rangle_{D(z_0,a)}|^2: \beta\in C^{\infty}_0(D(z_0,a)), 
\|\beta\|_{L^2(D(z_0,a),\lambda_n)}^2\leq 1 \right\}\\
\leq &\sup\left\{ |\langle u,(g_n)_{\zb}\rangle_{D(z_0,a)}|^2: 
\|g_n\|_{L^2(D(z_0,a),\lambda_n)}^2\leq 2\ep^2 \right\}\\
=&\sup\left\{ |\langle u_z,g_n \rangle_{D(z_0,a)}|^2: 
\|g_n\|_{L^2(D(z_0,a),\lambda_n)}^2\leq 2\ep^2 \right\}\\
\leq &2\ep^2 \int_{D(z_0,a)}|u_z(z)|^2e^{\lambda_n(z)}dV(z).
\end{align*}
There exists $0<c<1$ such that $e^{-\varphi(z)}<c e^{-\alpha(z)}$
for $z\in D(z_0,a)$. Then 
\[\frac{\pi}{n-1} e^{(2n-2)\varphi(z)}(1-c^{2n-2})<e^{-\lambda_n(z)}<\frac{\pi}{n-1} 
e^{(2n-2)\varphi(z)}.\]
So for large $n$ we have 
\[\frac{\pi}{2(n-1)} e^{(2n-2)\varphi(z)} <e^{-\lambda_n(z)}
<\frac{\pi}{n-1} e^{(2n-2)\varphi(z)}\]
and 
\begin{align*}
 \frac{n-1}{\pi}\int_{D(z_0,a)}|u(z)|^2e^{(2-2n)\varphi(z)}dV(z)
<&\int_{D(z_0,a)}|u(z)|^2e^{\lambda_n(z)}dV(z) \\
\leq & 2\ep^2 \int_{D(z_0,a)}|u_z(z)|^2e^{\lambda_n(z)}dV(z)\\
\leq& \frac{4\ep^2(n-1)}{\pi} \int_{D(z_0,a)}|u_z(z)|^2e^{(2-2n)\varphi(z)}dV(z).
\end{align*}
That is, for any $\ep>0$ and  $u\in C^{\infty}_0(D(z_0,a))$ 
\begin{equation}\label{estimate1}
\int_{D(z_0,a)}|u(z)|^2e^{(2-2n)\varphi(z)}dV(z)\leq 4\ep^2  
\int_{D(z_0,a)}|u_z(z)|^2e^{(2-2n)\varphi(z)}dV(z)
\end{equation}
for large $n$.

The estimate in \eqref{estimate1} is identical to the one in 
\cite[pg. 38, proof of Lemma 10.2]{FuChrist05}. 
That is $\lambda^m_{n\varphi}(D(z_0,a))\to \infty$ as $n\to \infty$ 
(see \cite[Definition 2.3]{FuChrist05}). 
Since $\varphi$ is smooth and subharmonic, \cite[Theorem 1.5]{FuChrist05} 
implies that $\lambda^e_{n\varphi}(D(z_0,a))\to \infty$ as $n\to \infty$. 
We note that \cite[Theorem 1.5]{FuChrist05} implies that 
if $\lambda^m_{n\varphi}(D(z_0,a))\to \infty$ as $n\to \infty$ then 
$\lambda^e_{n\varphi}(D(z_0,a))\to \infty$ as $n\to \infty$. 
This is enough to conclude that $\Gamma_k$ is $B$-regular. This 
argument is contained in  the proof of Proposition 9.1 converse 
of (1) in \cite[pg 33]{FuChrist05}. 
We repeat the argument here for the convenience of the reader. 
Let $V=\{z\in D(z_0,a):\Delta\varphi(z)>0\}$ and 
$K_0=\overline{D(z_0,a/2)}\setminus V$. Then $V$ is open and 
$K_0$ is a compact subset of $D(z_0,a)$. Furthermore, $\Delta\varphi=0$ 
on $K_0$. If $K_0$ has non-trivial fine interior then it supports a nonzero 
function $f\in W^1(\C)$ (see \cite[Proposition 4.17]{StraubeBook}). 
Then  
\[\lambda^e_{n\varphi}(D(z_0,a))\leq \frac{\|\nabla f\|^2}{\|f\|^2}
<\infty \text{ for all } n.\]
Which is a contradiction. Hence $K_0$ has empty fine interior which implies 
that $K_0$ satisfies property (P) (see \cite[Proposition 4.17]{StraubeBook} or  
\cite[Proposition 1.11]{Sibony87}). Therefore, for $M>0$ there exists an open 
neighborhood $O_M$ of $K_0$ and $b_M\in C^{\infty}_0(O_M)$ such that 
$|b_M|\leq 1/2$ on $O_M$ and $\Delta b_M>M$ on $K_0$. Furthermore, 
using the assumption that $|w|>0$ on $\Gamma_k$ one can choose $M_1$ such that 
the function $g_{M_1}(z,w)=M_1(|w|^2e^{\varphi(z)}-1)+b_M(z)$ has the following 
properties: $ |g_{M_1}| \leq1$ and the complex Hessian $H_{g_{M_1}}(W)\geq M\|W\|^2$ 
 on $\Gamma_k\cap \overline{D(z_0,a)}$ where $W$ is complex tangential direction. 
 Then \cite[Proposition 3.1.7]{AyyuruThesis} implies that  
 $\Gamma_k\cap \overline{D(z_0,a/2)}$ satisfies property (P) (hence it is $B$-regular).
Therefore, \cite[Corollary 4.13]{StraubeBook} implies that $\Gamma_k$ is $B$-regular. 

The computations in the second case (that is $\rho_{|w|}(z_0,|w_0|)>0$) are very similar. 
So we will just highlight the differences between the two cases. We define 
\[U^{a_1,b_1}=\left\{(z,w)\in \C^2: |w|>b_1|z-z_0|^2+a_1\right\}\]
and 
\[U_1=\D\cap U^{a_1,b_1}
=\left\{(z,w)\in \C^2:z\in V_1,e^{-\alpha(z)} <|w|<e^{-\varphi(z)}\right\}\]
where $V_1$ is a domain in $\C$ and where $\alpha(z)=-\log(b_1|z-z_0|^2+a_1)$ 
is a strictly superharmonic function. One can show that $bU^{a_1,b_1}$ is strongly 
pseudoconvex.  We choose $a,a_1,b,b_1>0$ such that  such that 
$\overline{D(z_0,a)}\subset V_1$ and $U$ is given by
\begin{align*}
U=\D\cap U_{a,b}
= \left\{(z,w)\in \C^2:z\in D(z_0,a), e^{-\alpha(z)} <|w|<e^{-\varphi(z)}\right\}
\end{align*}
where $U_{a,b}=D(z_0,a)\times \{w\in \C:|w_0|-b<|w|<|w_0|+b\}$. 
Furthermore, we define 
\[\lambda_n(z)=-\log\left(\frac{\pi}{n+1}\left(e^{-(2n+2)\varphi(z)}
-e^{-(2n+2)\alpha(z)}\right)\right)\]
 for $n=0,1,2,\ldots$ and by scaling $U_1$ in $w$ variable if necessary, 
 we will assume that $U_1\subset D(z_0,a_1)\times \{w\in \C:|w|<1\}$ so that 
$\|1\|_{L^2(D(z_0,a_1),\lambda_n)}$ goes to zero as $n\to\infty$. 
One can check that $\lambda_n$ is subharmonic for all $n$.

We take functions $\beta\in  C^{\infty}_0(D(z_0,a))$ and consider symbols 
$\psi\in C^{\infty}(\overline{V_1})$ such that $\psi_{\overline{z}}=\beta$. 
Then we consider the functions $H_{\psi}w^n$ for $n=0,1,2,\ldots$. 
Calculations similar to the ones in the previous case  reveal that 
$g_n(z)w^n=H_{\psi}w^n$ where  $g_n=S^{V_1}_{\lambda_n}(\beta d\zb)$. 
Using similar manipulations and again the compactness estimate 
\eqref{EqnCompEst} we conclude that for any $\ep>0$ there exists an integer 
$n_{\ep}$ such that for $u\in C^{\infty}_0(D(z_0,a))$ and $n\geq n_{\ep}$ we have  
\begin{equation}\label{estimate2}
\int_{D(z_0,a)}|u(z)|^2e^{(2n+2)\varphi(z)}dV(z)\leq \ep  
\int_{D(z_0,a)}|u_z(z)|^2e^{(2n+2)\varphi(z)}dV(z).
\end{equation}
Finally, an argument similar to the one right after \eqref{estimate1} implies that 
$\Gamma_k$ is $B$-regular.
\end{proof}

\section*{Acknowledgement}
We would like to thank the anonymous referee for constructive comments.

\end{document}